\documentclass[12pt]{article}
\usepackage{mathrsfs}
\usepackage{amsfonts,amsthm, amssymb}
\usepackage[tbtags]{amsmath}
\usepackage{latexsym, euscript, epic, eepic}
\usepackage{lineno}
\usepackage{cite}

\newtheorem{defn}{Definition}[section]
\newtheorem{prop}{Proposition}[section]
\newtheorem{thm}{Theorem}[section]
\newtheorem{lem}{Lemma}[section]
\newtheorem{Example}{Example}[section]

\numberwithin{equation}{section}

\begin{document}

\title{On the topological entropy of the set with a special shadowing time
 \footnotetext {* Corresponding author}
  \footnotetext {MSC classes:28A80, 28D05, 11K55}}
\author{Cao Zhao$^{1}$ and Ercai Chen$^{*1,2}$\\
  \small   1 School of Mathematical Sciences and Institute of Mathematics, Nanjing Normal University,\\
   \small   Nanjing 210023, Jiangsu, P.R.China\\
    \small 2 Center of Nonlinear Science, Nanjing University,\\
     \small   Nanjing 210093, Jiangsu, P.R.China.\\
      \small    e-mail: izhaocao@126.com,  ecchen@njnu.edu.cn.
}
\date{}
\maketitle

\begin{center}
 \begin{minipage}{120mm}
{\small {\bf Abstract.}
Let  $(X,d,T )$ be  a topological dynamical system with specification property.
For  $ \alpha\in  \mathbb R^+$ and any $x_0\in X$, define
 $$ \mathbf D^{x_0}_\alpha  :=\Big\{x\in X:   \lim\limits_{\epsilon\to 0}\limsup\limits_{n\to\infty}\dfrac{\max\{t\in\mathbb N:~T^n(x)\in B_{t}( x_0,\epsilon)\}}{n}\geq \alpha\Big\}.$$
    Then we have
 $h_{top}^{B}  ( T, \mathbf D^{x_0}_\alpha )=  \dfrac{h_{top}(T)}{1+\alpha} $, where $h_{top}^{B}  ( T, \mathbf D^{x_0}_\alpha  )$ denotes the Bowen topological entropy of  $\mathbf D^{x_0}_\alpha.$}
\end{minipage}
 \end{center}

\vskip0.5cm
\section{Introduction}

Part of number theory is concerned with finding rational numbers $\dfrac{p}{q}$ whcih are good approximations to a real number $x$. For any $x$ one can find infinitely many $\dfrac{p}{q}$ whose distance from $x$ is less than $\dfrac{1}{q^2}.$ If one can find infinitely many $\dfrac{p}{q}$ whose
distance from $x$ is less than $\dfrac{1}{q^{\tau}}$ with $\tau>2$ then $x$ is a $\tau$-well approximable number.

Let $\Psi: \mathbb R^+\to \mathbb R^+$ be a decreasing function such that $x\mapsto x^2\Psi(x)$ is nonincreasing. In 1924, Khintchine\cite{Khi} used the theory of continued fraction to prove that the set of $\Psi$-approximable real numbers 
$$\mathcal{K}(\Psi):= \left\{ \left|\xi -\frac{p}{q}\right|<\Psi(q)~\text{for infinitely many rational numbers}~\frac{p}{q}\right\}$$
have Lebesgue meausure zero if the sum $\sum_{x\geq 1}x\Psi(x)$ converges and has full Lebesgue measure otherwise.  Furthermore,  this result has been refined by 
Jarn\'ik \cite{Jar} and, independently, Besicovitch \cite{Bes}, who established that, for any real number $\tau\geq 1$, the Hasudorff dimension of the set 
$\mathcal{K}(\tau):= \mathcal{K}(x\mapsto x^{-2\tau})$ is equal to $1/\tau.$
Moreover, in 2003, Bugaud \cite{Bug} make the precise research on the exact approximation by adjusting the subset   $\mathcal{K}(\Psi)$. 
In \cite{Hil, Hil1,Hil2}, Hill and Velani introduced the shrinking target problems sourced from number theory and gave quantitative research  of the recurrence.
Let $T:\mathbf J\to \mathbf J$ be an expanding rational map of the Riemann sphere acting on
its Julia set $\mathbf J$ and $f : \mathbf J \to  \mathbb R $ denote  a H\"{o}lder continuous function satisfying $f(x)\geq
\log |T^{'}(x)|$ for all $x\in \mathbf J$. For any $z_{0}\in \mathbf J$,  Hill and Velani \cite{Hil} studied the set
of "well approximable" points
$$D_{z_{0}}(f) := \{x\in \mathbf J  : d(y, x) < e^{-S_{n}f(y)}  ~\text{for infinitely many pairs }~ (y, n)\in \mathbf I\},$$
where $\mathbf I = \mathbf I(z_{0}) $ denotes the set of pairs $(y, n)(n\in \mathbb N)$ such that $T^{n}y=z_{0}$ and $S_{n}f(y)=\sum_{i=0}^{n-1}f(T^{i}y)$. In fact,
they gave the following result.

{\bf Theorem A}
The set $D_{z_{0}}(f)$
  has Hausdorff dimension $s(f)$, where $s(f)$ is the unique
solution to the pressure equation
$$P(-s f )= 0.$$
In \cite{TanWan}, Tan and Wang investigated metric properties as well as estimates on the Hausdorff dimension of the recurrence set for $\beta$-transformation dynamical systems. More precisely,
 the $\beta$-transformation $T_{\beta}:[0,1]\to [0,1]$ is defined by $T_{\beta}x=\beta x- \lfloor  \beta x\rfloor $  for all $x\in [0,1]$. The spotlight is on the size of the set
$$ \{x \in [0, 1] : d(T_{\beta}^{n}x, x)<\psi(n) ~\text{for infinitely many}~ n\},  $$
where $\psi(n)$ is a positive function. In fact, this has evoked a rich subsequent literature on the so-called Diophantine approximation. We refer the reader to \cite{Wan, Fan} for the related work about this set.

Inspired by the research  on  "recurrence behaviour" or "shrinking behaviour" in dynamical systems. In this paper, we define a class sets with a special shadowing behaviour.  Let $(X, d, T)$ be a topological dynamical system, where $(X,d)$ is a compact metric space and $T:X\to X$ is a continuous map. 
For $n\in \mathbb{N}$, the Bowen metric $d_n$ on $X$  is defined by
\begin{align*}
d_n(x,y):=\max_{0\leq i\leq n-1}d(T^i(x),T^i(y)).
\end{align*}
For every $\epsilon>0$, $n\in \mathbb{N}$ and a point $x\in X$, define the Bowen ball
$$B_n(  x,\epsilon):=\Big\{y\in X: d_{n}(x, y)<\epsilon\Big\}.$$
 Let   $\alpha\in \mathbb R^+$ and $x_0\in X$  we define the level set
 $$ \mathbf D^{x_0}_\alpha  :=\Big\{x\in X:   \lim\limits_{\epsilon\to 0}\limsup\limits_{n\to\infty}\dfrac{\max\{t\in\mathbb N:~T^n(x)\in B_{t}( x_0,\epsilon)\}}{n}\geq \alpha\Big\} .$$

Now we state our main results as follows. (See Section \ref{sep} for precise definitions.)
\begin{thm}\label{main0}
Let  $T: X\to X$ be a continuous map with specification property. For any
 $\alpha\in \mathbb R^+$ and $x_0\in X$,
$$h_{top}^B  (  \mathbf D^{x_0}_\alpha   )=  \dfrac{h_{top}(T)}{1+\alpha} ,$$
where $h_{top}^B  ( \mathbf D^{x_0}_\alpha   )$ denotes the Bowen topological entropy of the subset $\mathbf D^{x_0}_\alpha .$
\end{thm}
In fact, in this paper we give a more general result than the above theorem.
For any positive continuous map $f: X\to \mathbb R$,  and $x_{0}\in X,$  we define
$$ \mathbf D^{x_0}_f :=\Big\{x\in X:   \lim\limits_{\epsilon\to 0}\limsup\limits_{n\to\infty}\dfrac{\max\{t\in\mathbb N:~T^n(x)\in B_{t}( x_0,\epsilon)\}-S_nf(x)}{n}\geq 0\Big\} .$$
\begin{thm}\label{main}
Let  $T: X\to X$ be a continuous map with specification property.   Let $f : X \to \mathbb R^+$ be a positive continuous function.  For each $x_{0}\in X,$ we have the topological entropy of the subset $\mathbf D^{x_0}_f ,$
 $$h_{top}^B  (T, \mathbf D^{x_0}_f  )=\sup   \Big\{\dfrac{h_{\mu}(T)}{1+    \int f d\mu}:\mu\in M(X, T)\Big\},$$
 which is the solution of the pressure equation $P(-s(f+1))=0.$
\end{thm}

\section{Preliminaries}\label{sep}
Let $(X, d, T)$ be a topological dynamical system, where $(X,d)$ is a compact metric space and $T:X\to X$ is a continuous map. The set $M(X)$ of all Borel probability measures is compact under the weak$^*$ topology. Denote by $M(X,T)\subset M(X)$ the subset of all $T$-invariant measures and $E(X,T)\subset M(X, T)$ the subset of all ergodic measures. 
Let $C(X)$ denote the space of continuous function from $X$ to $\mathbb R$. Let $S_{n}\varphi(x):= \sum_{i=0}^{n-1}\varphi(T^i(x))$
and for $c>0$, let
$$\text{Var}(\varphi, c):=\sup\{|\varphi(x)-\varphi(y)|:~d(x, y)<c\}.$$
  Given $\epsilon>0$ and $n\in \mathbb N$, we say that a set $E\subset X$ is $(n, \epsilon)$-spanning set, if for each $x\in X$ there exists $a\in E$ such that $d(T^ix, T^ia)<\epsilon$ for every $0\leq i\leq n-1$. In other words,
  $$X\subset \bigcup_{a\in E}B_{n}(a, \epsilon).$$
We also introduce  a dual notion of the spanning set. Given $\epsilon>0$ and $n\in \mathbb N$, we say that a set $E\subset X$ is $(n, \epsilon)$-separated if, given $x, y\in E$, there exists $0\leq j\leq n-1$ such that $d(T^jx, T^jy)\geq \epsilon$. In other words, if $x\in E$ then $B_{n}(x, \epsilon)$ contains no other point of $E$.
In the following, we give the definition of classical topological pressure.
\begin{defn}
For any $\varphi\in C(X), n\geq 1$ and $\epsilon>0$ put
$$Q_{n}(T, \varphi, \epsilon)=\inf\Big\{\sum_{x\in F}\exp \Big(\inf_{y\in B_{n}(x, \epsilon)} S_n\varphi(y)\Big)~: F~\text{is a }~ (n ,\epsilon)~\text{spanning set for }~X\Big\}.$$
Put $Q(T, \varphi, \epsilon)=\limsup\limits_{n\to \infty}\frac{1}{n}\log Q_n(T, \varphi, \epsilon).$
We define the topological pressure $$P(\varphi)=\lim\limits_{\epsilon\to 0}Q(T, \varphi, \epsilon).$$
If $\varphi=0$, we define the topological entropy $h_{top}(T)=P(0)$.
\end{defn}
 Next we give the definition of topological pressure as a characteristic of dimension type, due to Pesin and  Pitskel \cite{PesPit}. The definition generalises Bowen's definition of topological entropy for non-compact sets in \cite{Bow}.
\begin{defn}{\rm (\cite{PesPit})}
Let $Z\subset X$ be an arbitrary Borel set, not necessarily compact or invariant. Firstly,  we consider finite and countable collections of the form $\Gamma=\{B_{n_i}(x_i, \epsilon)\}_i$. For $s\in \mathbb R$ and $\varphi\in C(X)$, we define the following quantities:
$$M(\varphi, Z, s, \Gamma  )=\sum_{B_{n_i}(x_i, \epsilon)\in \Gamma}\exp\Big(-sn_i+\sup\limits_{x\in B_{n_i}(x_i, \epsilon)}\sum_{k=0}^{n_i-1}\varphi(T^k(x))\Big)$$
$$M(\varphi, Z, s, \epsilon, N  )=\inf_{\Gamma}M(\varphi, Z, s, \Gamma  )$$
where the infimum is taken over all finite or countable collections of the form $\Gamma=\{B_{n_i}(x_i, \epsilon)\}_i$ with $x_i\in X$ such that $\Gamma$ covers $Z$ and $n_i\geq N$ for all $i=1, 2,\cdots$. Define
$$ M(\varphi, Z, s, \epsilon  )=\lim\limits_{N\to\infty}M(\varphi, Z, s, \epsilon, N ).$$
The existence of the limit is guaranteed since the function $M(\varphi, Z, s, \epsilon, N  )$ does not decrease with $N$. By standard techniques, we can show the existence of
$$P (\varphi,Z, \epsilon):=\inf\{s: M(\varphi, Z, s, \epsilon  )=0\}=\sup\{s:M(\varphi, Z, s, \epsilon )=\infty\}.$$
Finally, define $$P ( \varphi,Z):=\lim\limits_{\epsilon\to 0}P (\varphi, Z,  \epsilon).$$
If $\varphi=0$, we define the Bowen topological entropy $h_{top}^B(  Z)=P ( 0,Z) $ with $h_{top}^B(  Z, \epsilon)=P ( 0,Z, \epsilon) $.
\end{defn}
In the end of this section,  we give an alternative formulation of Pesin and Pitskel's topological pressure. We generalize  the definition of Bowen balls.
For any $x_0, x_1, \cdots, x_{n-1}\in X$ and $\epsilon>0$, we define  $$B(x_0, x_1, \cdots, x_{n-1}; \epsilon):=\{x\in X: d(T^ix, x_i)<\epsilon, 0\leq i\leq n-1\}.$$

 \begin{defn}\label{lllll}
Let $Z\subset X$ be arbitrary Borel set, not necessarily compact or invariant. We consider finite and countable collections of the form $\Lambda=\{B(x_0, x_1, \cdots, x_{n_i-1}; \epsilon)\}_i$. For $s\in \mathbb R$ and $\varphi\in C(X)$, we define the following quantities:
$$\widehat{M}( \varphi, Z, s, \Lambda )=\sum_{B(x_0, x_1, \cdots, x_{n_i-1}; \epsilon)\in \Lambda}\exp\Big(-sn_i+\sup\limits_{x\in B(x_0, x_1, \cdots, x_{n_i-1}; \epsilon)}\sum_{k=0}^{n_i-1}\varphi(T^k(x))\Big)$$
$$\widehat{M}( \varphi, Z, s, \epsilon, N )=\inf_{\Lambda}\widehat{M}( \varphi, Z, s, \Lambda ),$$
where the infimum is taken over all finite or countable collections of the form $\Lambda=\{B(x_0, x_1, \cdots, x_{n_i-1}; \epsilon)\}_i$ with $x_i\in X$ such that $\Lambda$ covers $Z$ and $n_i\geq N$ for all $i=0, 1,\cdots$. Define
$$\widehat{M}( \varphi, Z, s, \epsilon )=\lim\limits_{N\to\infty}\widehat{M}( \varphi, Z, s, \epsilon, N ).$$
The existence of the limit is guaranteed since the function $\widehat{M}( \varphi, Z, s, \epsilon, N )$ does not decrease with $N$. By standard techniques, we can show the existence of
$$\widehat{P} (  \varphi, Z, \epsilon):=\inf\{s:\widehat{M}( \varphi, Z, s, \epsilon  )=0\}=\sup\{s:\widehat{M}( \varphi,Z, s, \epsilon )=\infty\}.$$
Finally, define $$\widehat{P} (  \varphi, Z   ):=\lim\limits_{\epsilon\to 0}\widehat{P} (  \varphi, Z, \epsilon).$$
\end{defn}
 \begin{prop}
 For any $\varphi \in C(X)$, we have
 $$P (  \varphi, Z) =\widehat{P} (  \varphi, Z).$$
 \begin{proof}
 Assume that $\Gamma=\{B_{n_i}(x_i, \epsilon)\}_i$ covers $Z$, we define   $\Lambda(\Gamma):=\{B(x_0, x_1, \cdots, x_{n_i-1}; \epsilon)\}_i$ with $x_0=\cdots =x_{n_i-1}=x_i.$
 Then $\widehat{M}(\varphi, Z, s, \epsilon, N  )\leq  M( \varphi, Z, s, \epsilon, N ) $ for any $   s>0, \epsilon>0$ and $N\in \mathbb N.$ Hence, $\widehat{P} (\varphi, Z) \leq P (\varphi, Z).$ Next we prove
 $$\widehat{P} (\varphi, Z) \geq P (\varphi, Z).$$
 Fix $t>\widehat{P} (\varphi,Z)$. We set $2\epsilon_0:=t-\widehat{P} (\varphi, Z)$. Choose $\epsilon$   small enough such that
 $$\text{Var}(\varphi,  4\epsilon)\leq \epsilon_0.$$
 There exists $N>0$ and a
 finite and countable collections of the form $$\Lambda=\{B(x_0, x_1, \cdots, x_{n_i-1}; \epsilon)\}_i$$   with  $n_i\geq N$, $x_0, x_1, \cdots, x_{n_i-1}\in X $ and $\Lambda$ covers $Z$ such that
\begin{align}\label{defeq}
\begin{split}&\widehat{M}(\varphi, Z, t-\epsilon_0, \Lambda  )\\
\leq &\sum_{B(x_0, x_1, \cdots, x_{n_i-1}; \epsilon)\in \Lambda}\exp\Big(-(t-\epsilon_0)n_i+\sup\limits_{x\in B(x_0, x_1, \cdots, x_{n_i-1}; \epsilon)} S_{n_i}\varphi(x)\Big) \\
\leq& 1.
\end{split}
\end{align}
For each $B(x_0, x_1, \cdots, x_{n_i-1}; \epsilon)\in \Lambda$, we choose $y_i \in B(x_0, x_1, \cdots, x_{n_i-1}; \epsilon)$, and define $B_{n_i }(y_i  , 2\epsilon)$. Then $B_{n_i}(y_i , 2\epsilon)\supset B(x_0, x_1, \cdots, x_{n_i-1}; \epsilon)$.
Hence, we can pick the respect $\Gamma =\{B_{n_i}(y_i , 2\epsilon)\}$ covers $Z$. By (\ref{defeq}), one has
\begin{align*}
\begin{split}
&M( \varphi, Z, t, \Gamma )\\
\leq &\sum_{B_{n_i}(y_i, \epsilon)\in \Gamma}\exp\Big(-tn_i+\sup\limits_{x\in B_{n_i}(y_i, \epsilon)}S_{n_i}\varphi(x)\Big)
\\\leq &\sum_{B(x_0, x_1, \cdots, x_{n_i-1}; \epsilon)\in \Lambda}\exp\Big(-tn_i+\sup\limits_{x\in B(x_0, x_1, \cdots, x_{n_i-1}; \epsilon)}S_{n_i}\varphi(x)+n_i\text{Var}(\varphi, 4\epsilon)\Big)
\\
\leq &\sum_{B(x_0, x_1, \cdots, x_{n_i-1}; \epsilon)\in \Lambda}\exp\Big(-(t-\epsilon_0)n_i+\sup\limits_{x\in B(x_0, x_1, \cdots, x_{n_i-1}; \epsilon)} S_{n_i}\varphi(x)-n_i\epsilon_0+n_i \epsilon_0\Big) \\
\leq &1.
\end{split}
\end{align*}
Hence  $ P (\varphi,Z, 2\epsilon)\leq t.$ Letting $\epsilon\rightarrow0$,  $ P (\varphi , Z  )\leq t $.  Since $t$ can be  arbitrary close to $ P (\varphi , Z  )$, one has
 $$\widehat{P} (\varphi, Z) \geq P (\varphi, Z).$$

 \end{proof}
\end{prop}

In this paper, we study transformation $T$ of the following type:
\begin{defn}
A continuous map $T$ satisfies the specification property
if for each $\epsilon>0$ there is an integer
$m=m(\epsilon)$ such that for any collection $\{I_{j}=[a_{j},b_{j}]\subset\mathbb N:j=1,\cdots,k\}$ of finite intervals with $a_{j+1}-b_{j}\geq m(\epsilon)$ for $j=1,\cdots,k-1$ and any $x_{1},\cdots,x_{k}$ in $X$, there exists a point $x$ such that
$$d(T^{p+a_{j}}x,T^{p}x_{j})<\epsilon~\text{for all}~p=0,\cdots,b_{j}-a_{j}~\text{and every}~j=1,\cdots,k.$$
\end{defn}

\section{Proof of Theorem \ref{main}}\label{proofm}
In this section, we will verify Theorem \ref{main}.
The upper bound on $h_{top} ^{B} (   \mathbf D^{x_0}_f  )$ is easy to get by Definition \ref{lllll}.
To obtain the lower bound estimate we need to construct
a suitable  Moran set. The proof will be
divided into the following two subsections.
\subsection{Upper bound for $h_{top} ^{B} (   \mathbf D^{x_0}_f  )$}
For any positive continuous map $f: X\to \mathbb R$, and $x_{0}\in X$. We recall that
$$ \mathbf D^{x_0}_f   :=\Big\{x\in X:  \lim\limits_{\epsilon\to 0}\limsup\limits_{n\to\infty}\dfrac{\max\{t\in\mathbb N:~T^n(x)\in B_{t}( x_0,\epsilon)\}-S_nf(x)}{n}\geq 0\Big\} .$$
 To prove upper bound for $h_{top}^{B}  (   \mathbf D^{x_0}_f  )$ , we only need to prove the following Theorem.
\begin{thm}
Let  $T: X\to X$ be a continuous map. For any positive continuous function $f\in C(X)$ and $x_0\in X$, we have the topological entropy of the subset $\mathbf D^{x_0}_f  ,$
$$h_{top} ^{B }(   \mathbf D^{x_0}_f    )\leq  \sup   \Big\{\dfrac{h_{\mu}(T)}{1+
\int f d\mu}:\mu\in M(X, T)\Big\}  ,$$
 which is the unique solution of topological pressure $P(-s(f+1))=0$.
\begin{proof}
For each $\epsilon>0$, define
$$\mathbf D^{x_0, \epsilon}_f    :=\Big\{x\in X:  \limsup\limits_{n\to\infty}\dfrac{\max\{t\in\mathbb N:~T^n(x)\in B_{t}( x_0,\epsilon)\}-S_nf(x)}{n}\geq 0\Big\} .$$
Clearly, $$\mathbf D^{x_0 }_f\subset \bigcap_{\epsilon>0}\mathbf D^{x_0, \epsilon}_f.$$
Let  $s_0=\sup   \Big\{\dfrac{h_{\mu}(T)}{1+
\int f d\mu}:\mu\in M(X, T)\Big\}$  be the unique solution of topological pressure $P(-s(f+1))=0$.
Fix any $s>s_0$.
It is easy to check that $P:=P(-s (f+1))<0$.
Then, there exists   $\epsilon_1>0, N_1>0$  such that, for any $\eta<\epsilon_1, n\geq N_1$,
 \begin{align}\label{eeta}
 s \eta<-\frac{P}{4},
 \end{align}
 and we can choose the $(n, \eta)$-spanning set $E_{n,\eta}$ such that
\begin{align}\label{ennt}
\sum_{x\in E_{n,\eta}}e^{- s  \inf\limits_{y\in B_{n}(x,\eta)}S_{n}(f+1)(y)}\leq e^{ \frac{nP}{2}}.
\end{align}
Furthermore,
 \begin{align}\label{llll}
 \mathbf D^{x_0}_f\subset \mathbf D^{x_0, \eta }_f\subset  \bigcap_{N=1}^{\infty} \bigcup_{n=N}^{\infty}\bigcup_{ y\in E_{n,\eta }}J( y,    \eta, n, f),
 \end{align}
 where
 \begin{align*}
 \begin{split}
 &J( y,   \eta, n, f) \\
 =&\Bigg\{x\in B_{n}(y, \eta):  \dfrac{\max\{t\in \mathbb N:~ T^nx\in B_{t}(x_0, \eta )\}}{n}   >\dfrac{\inf\limits_{x\in B_{n}(y, \eta)}S_nf(x)}{n} -\eta\Bigg\}  .
 \end{split}
 \end{align*}
 For each $n\geq 1$ and each $y\in E_{n,\eta}$, we define  $t_n^y:=\lfloor  \inf\limits_{x\in B_{n}(y, \eta)}S_nf(x) -n\eta \rfloor $.
For each $x\in   J( y,   \eta, n, f),$
\begin{align*}
  T^nx \in B_{t^y_n  }(  x_0,  \eta),
 \end{align*}
which implies that
 $$J( y, \eta, n,f)\subset B_{n}(y,  \eta)\cap T^{-n} B_{t_n^y}( x_0,   \eta ). $$
So we can get a family of covers $ \Lambda (\eta)=\{B_{n}(y,  \eta)\cap T^{-n} B_{t^y_n}( x_0,  \eta)\}_{y\in E_{n,\eta}, n\geq N_1}.$  For each $B_{n}(y,  \eta)\cap T^{-n} B_{t^y_n}( x_0,  \eta)$,  we define
$y_0=y, y_1=Ty, \cdots, y_{n-1}=T^{n-1}y$ and $y_{n}=x_0, y_{n+1}=T (x_0), \cdots y_{n+t_{n}^y-1}= T^{t_n^y-1}(x_0)$, and get $B(y_0, \cdots, y_{n+t_{n}^y}; \eta).$
Thus we  have
$$B_{n}(y,  \epsilon)\cap T^{-n} B_{ t^y_n}( x_0,  \eta)=B(y_0, \cdots ,y_{n+t_{n}^y-1}; \eta).$$
Hence, we  get the family $\Lambda(\eta):=\{B(y_1, \cdots ,y_{n+t_{n}^y}; \eta)\}$ with $n\geq N_{1}$ covers $Z$.  By (\ref{llll}), (\ref{ennt}) and (\ref{eeta}), one has
 \begin{align*}
 \begin{split}
 & \widehat{M}( 0,\mathbf D^{x_0 }_f, s,    \Lambda(\eta), N_1 )\\
 \leq &\sum_{n=N_1}^{\infty} \sum_{y\in E_{n ,\eta}}e^{-s(n+t_n^y)}\\
  \leq & \sum_{n=N_1}^{\infty}\sum_{y\in E_{n ,\eta}}e^{-s(n+\inf\limits_{x\in B_{n}(y, \eta)}S_nf(x) -n\eta -1) } \\
 = & \sum_{n=N_1}^{\infty}\sum_{y\in E_{n ,\eta}}e^{- s  (n+\inf\limits_{x\in B_{n}(y, \eta)}S_nf(x)  )+s  n\eta +s   }\\
  \leq &\sum_{n=N_1}^{\infty}e^{   \frac{nP}{2} +s  n\eta +s }\\
 \leq &\sum_{n=N_1}^{\infty} e^{   \frac{nP}{4} +s } <\infty.
 \end{split}
 \end{align*}
 Hence, $ h_{top}^B(  \mathbf D^{x_0  }_f)=\lim\limits_{
 \eta\to 0} \widehat{P} (0,\mathbf D^{x_0, \eta }_f,  \eta)\leq s$. Let $s\to s_0$,  we finish the proof.
\end{proof}
\end{thm}

\subsection{Lower bound for $h^B_{top}  (    \mathbf D^{x_0 }_f )$}
Firstly, we introduce the Entropy Distribution Principle as follows.
\begin{prop}[\cite{Tho1} Entropy Distribution Principle ] Let $T:X\to X$ be a continuous transformation. Let $Z\subset X$ be arbitrary Borel set. Suppose there exists $\epsilon>0$ and $s\geq 0$ such that one can find a sequence of Borel probability measure $\mu_k$, a constant $K>0$ and an integer $N$ satisfying
$$\lim\limits_{k\to\infty}\sup(\mu_{k}(B_{n}(x, \epsilon)))\leq K\exp ( -sn) $$
for every ball $B_n(x, \epsilon)$ such that $B_n(x, \epsilon)\cap Z\neq \emptyset$ and $n\geq N.$ Furthermore, assume that at least one limit measure $\nu$ of the sequence $\mu_k$ satisfies $\nu(Z)>0$. Then $ h^B_{top}( Z, \epsilon)\geq s.$
 \end{prop}

The following result generalises Katok's formula for measure-theoretic entropy. In \cite{Men}, Mendoza gave a proof based on idea from the Misiurewicz proof of the variational principle. Although he states the result under the assumption that $T$ is a homeomorphism, his proof works for $T$ continuous.

  \begin{prop}\label{Katok}
  Let $(X, d)$ be a compact metric space $ T: X\to X$ be a continuous map and $\mu$ be an ergodic invariant measure. For $ \epsilon>0, \gamma\in (0, 1)$ and $\phi\in C(X)$, define
  $$N^{\mu}(\phi,\gamma, \epsilon, n)=\inf \Big\{\sum_{x\in \mathcal{S}}\Big\{\sum_{i=0}^{n-1}\phi(T^ix)\Big\}\Big\},$$
  where the infimum is taken over all sets $\mathcal{S}$ which $(n, \epsilon) $ span some $Z$ with $\mu(Z)>1-\gamma$. We have
  $$h_{\mu}(T)+\int\phi d\mu= \lim\limits_{\epsilon\to 0}\liminf\limits_{n\to\infty}\frac{1}{n}\log N^{\mu}(\phi,\gamma, \epsilon, n).$$
  The formula remains true if we replace the $\liminf$ by $\limsup.$
  \end{prop}

\begin{thm}
Let  $T: X\to X$ be a continuous map with specification property. For any positive continuous function $f\in C(X)$ and $x_0\in X$, we have the topological entropy of the subset $\mathbf D^{x_0 }_f ,$
$$h_{top}^{B}  (T, \mathbf D^{x_0 }_f  )\geq   \sup   \Big\{\dfrac{h_{\mu}(T)}{1+
\int f d\mu}:\mu\in M(X, T)\Big\}  .$$
\end{thm}
 We fix some $C<  \sup   \Big\{\dfrac{h_{\mu }(T)}{1+
\int f d\mu}:\mu\in M(X, T)\Big\}$, and $\eta>0$.
 By ergodic decomposition Theorem, we can  choose  an ergodic measure $\mu\in E(X, T)$ such that
$$   \dfrac{h_{\mu }(T)}{1+
\int f d\mu} \geq C.$$
   For any $k\geq 1$, one can pick sufficiently large   $s_k $ so the set
$$Y_k:=\Big\{x\in X:~\left|\frac{S_{n}f(x)}{n}-\int f d\mu \right|<\eta  ~\text{for all}~n\geq s_k\Big\}$$
 satisfies $\mu (Y_k)>1-\eta$.
 This is possible by Birkhoff ergodic theorem.
\begin{prop}\label{propk}
For the above  $\eta>0$, we choose $ \epsilon>0$, a sequence $\{n_k\}_{k\geq 1} $
and  the  $(n_k, 6\epsilon)$-separated set $\mathcal{S}_k$ such that
\begin{align}\label{mma}
 M_k:=\sum_{x\in \mathcal{S}_k}e^{S_{n_k}(f+1)(x)}\geq  \exp (n_k(h_{\mu}(T)+1+\int f d\mu-\eta)).
\end{align}
Furthermore,
define $\epsilon_k:=\frac{\epsilon}{2^k} .$
Use the definition of specification property, we set $m_k:=m(\epsilon_k)$.
We can choose the above sequence $\{n_k\}_{k\geq 1}$ satisfies
 \begin{align}\label{et3}
\frac{m_k\|f\|}{n_k}<\frac{\eta}{4}.
\end{align}
\begin{proof}
By Proposition \ref{Katok}, we choose $\epsilon$ small enough such that
$$\liminf\limits_{n\to\infty}\frac{1}{n}\log N_n^{\mu}(f+1,\eta, 6\epsilon )\geq  h_{\mu}(T)+1+\int f d\mu-\frac{\eta}{2}.$$
For $A\subset X$, we define
\begin{align*}
Q_n(A, \epsilon)=\inf\left\{\# S:S ~~is~~
(  n,\epsilon)~~ spanning~~ set~~ for ~~A \right\},
\end{align*}
\begin{align*}
P_n(A, \epsilon)=\sup\left\{\# S:S ~~is~~
( n,\epsilon)~~ separated~~ set~~ for ~~A \right\}.
\end{align*}
Since $ \mu (Y_{k})>1-\eta$, we have
\begin{align*}
P_n(Y_{k }, 6\epsilon)\geq Q_n(Y_{k }, 6\epsilon)\geq N_n^{\mu}(f+1,\eta, 6\epsilon ).
\end{align*}
 By the definition of $N_n^{\mu}(f+1, \eta, 6\epsilon)$, we can choose a sequence $n_k\to\infty$ satisfying  (\ref{mma}),
  (\ref{et3}).
\end{proof}
\end{prop}
Without losing generality, in the following, we set the $\epsilon$ is small enough such that
\begin{align}\label{et1}
\text{Var}(f;  2\epsilon)<\frac{\eta}{4}.
\end{align}

\subsection{Construction of the Fractal $\mathbf F$}
We enumerate the points in the sets $\mathcal{S}_k$ provided by Proposition \ref{propk}  and write them as follows
$$\mathcal{S}_k=\{x_i : i=1, 2, \cdots, \# \mathcal{S}_k\}.$$
Let us choose any $N_1$ and define $l_1:= N_1(n_1+m_1)$.
For each $\underline{x}_1 =(x_1, \cdots, x_{N_1})\in  \mathcal{S}_1^{N_1}$.We
choose   $R_1 \in\mathbb{N} $ large enough so that
 \begin{align}
  R_1\geq \sup_{x\in H(\underline{x}_1;  2\epsilon_1 )}S_{l_1-m_1}f(x) \geq R_1-1,
 \end{align}
 where  $$H(\underline{x}_1; 2\epsilon_1 ):=  \bigcap_{j=1}^{N_1} T^{-(j-1)(n_1+m_1)}B_{n_{1}}(x_{j}, 2\epsilon_1  ) .$$
Using the specification property,
 we have
 $$D(\underline{x}_1; \epsilon_1 )=\bigcap_{j=1}^{N_1} T^{-(j-1)(n_1+m_1)}B_{n_{1}}(x_{ j }, \epsilon_1  )\bigcap T^{-l_1}B_{R_1}(x_0, \epsilon_1)\\
 =H(\underline{x}_1; \epsilon_1 )\bigcap T^{-l_1}B_{R_1}(x_0, \epsilon_1)\neq \emptyset.$$
 For each $\underline{x}_1\in  \mathcal{S}_1^{N_1}$, we choose a point $z(\underline{x}_1)\in D(\underline{x}_1; \epsilon_1 ).$
 Define the intermediate set
 $$\mathcal{C}_1:=\{z(\underline{x}_1):\underline{x}_1  \in  \mathcal{S}_1^{N_1} \}.$$
 For each $k\geq 1.$ We define the intermediate set $\mathcal{C}_{k+1}$ from $\mathcal{C}_k$. We recall that $\mathcal{C}_i, N_i, R_i, l_i$ for $1\leq i\leq k$ has been defined. For any $z(\underline{x}_{1},\cdots,\underline{x}_{k})\in \mathcal{C}_k$,
 we can choose $N_{k+1}$ large enough and define $l_{k+1}:=l_k+R_k+m_{k+1}+N_{k+1}(m_{k+1}+n_{k+1})$  such that
 \begin{align}\label{et2}
\frac{1}{N_{k+1}}<\frac{\eta}{4},
\end{align}
\begin{align}\label{et4}
\frac{(l_k+R_k)\|f\|}{l_{k+1}}<\frac{\eta}{4},
\end{align}
\begin{align}\label{ll}
\frac{\sum_{i=1}^{k+1}l_i}{l_k}<2,
\end{align}
and
\begin{align}\label{ll1}
\frac{n_{k+1}}{N_k}\rightarrow0
\end{align}
as $k\rightarrow\infty$. Now we can see that the chosen of $N_{k+1}$ may depend on $z(\underline{x}_{1},\cdots,\underline{x}_{k})$. But, as $(\underline{x}_{1},\cdots,\underline{x}_{k})\in \mathcal{S}_{1}^{N_{1}}\times \cdots \times\mathcal{S}_{k}^{N_{k}}$ and $\#(\mathcal{S}_{1}^{N_{1}}\times \cdots \times\mathcal{S}_{k}^{N_{k}})$ is finite, we can choose $N_{k+1}$ large enough such that $N_{k+1}$ dose not depend on $z(\underline{x}_{1},\cdots,\underline{x}_{k})$.

For each $\underline{x}_{k+1}=(x_1, \cdots, x_{N_{k+1}})\in  \mathcal{S}_{k+1}^{N_{k+1}}$.
We choose $R_{k+1} \in \mathbb N$ so that
\begin{align}\label{rrk}
  R_{k+1}\geq \sup_{x\in H(\underline{x}_1, \cdots, \underline{x}_k, \underline{x}_{k+1} ;  2\epsilon_{k+1} )}S_{l_{k+1}-m_{k+1}}f(x)\geq R_{k+1}- 1 ,
 \end{align}
 where
 \begin{align*}
 \begin{split}
 &H(\underline{x}_1, \cdots, \underline{x}_k, \underline{x}_{k+1} ;   2\epsilon_{k+1} ) \\
 =& B_{l_k+R_k} (  z(\underline{x}_1, \cdots, \underline{x}_k),   2\epsilon_{k+1} )\bigcap T^{-l_k-R_k-m_{k+1}}\Big(
 \cap_{j=1}^{N_{k+1}} T^{-(j-1)(n_{k+1}+m_{k+1})}B_{n_{k+1}}(x_{ j },2\epsilon_{k+1}  )\Big).
 \end{split}
 \end{align*}
 By the specification, we obtain
 $$D(\underline{x}_1, \cdots, \underline{x}_k, \underline{x}_{k+1} ;   \epsilon_{k+1} ):=H(\underline{x}_1, \cdots, \underline{x}_k, \underline{x}_{k+1} ;   \epsilon_{k+1} )\bigcap T^{-l_{k+1}}B_{R_{k+1}}(x_0, \epsilon_{k+1})\neq \emptyset.$$
Choose a point $z(\underline{x}_1, \cdots, \underline{x}_k, \underline{x}_{k+1}   )\in D(\underline{x}_1, \cdots, \underline{x}_k, \underline{x}_{k+1} ;   \epsilon_{k+1} )$. We define
\begin{align*}
\begin{split}
 \mathcal{C}_{k+1}
(z(\underline{x}_1, \cdots, \underline{x}_k   )) := \{ z(\underline{x}_1, \cdots, \underline{x}_k, \underline{x}_{k+1}   ): \underline{x}_{k+1}\in   \mathcal{S}_{k+1}^{N_{k+1}}\}
\end{split}
\end{align*}
for each $z(\underline{x}_1, \cdots, \underline{x}_k)\in \mathcal{C}_{k}.$  Moreover,
$$\mathcal{C}_{k+1}=\bigcup_{z(\underline{x}_1, \cdots, \underline{x}_k )\in\mathcal{C}_{k}}\mathcal{C}_{k+1}
(z(\underline{x}_1, \cdots, \underline{x}_k )).$$
In the following, we call $z\in\mathcal{C}_{k+1}$ descends from $x\in\mathcal{C}_{k}$ if $z=z(\underline{x}_1, \cdots, \underline{x}_k, \underline{x}_{k+1}   )$ and $x=z(\underline{x}_1, \cdots, \underline{x}_k)$ for some $(\underline{x}_{1},\cdots,\underline{x}_{k},\underline{x}_{k+1})\in \mathcal{S}_{1}^{N_{1}}\times \cdots \times\mathcal{S}_{k}^{N_{k}}\times\mathcal{S}_{k+1}^{N_{k+1}}$. We call the points in  $\mathcal{C}_{k+1}
(z(\underline{x}_1, \cdots, \underline{x}_k ))$ which descend from the same 'mother'  $(z(\underline{x}_1, \cdots, \underline{x}_k ))$ are 'sisters'.

In the end of the construction of intermediate $\{\mathcal{C}_k\}_{k\in \mathbb N}$, we remark with the estimates of the  time $R_k$ shadowing $x_0$.
\begin{prop}\label{rkupp}
For each $k\geq0$ and every $(\underline{x}_1, \cdots, \underline{x}_k, \underline{x}_{k+1})     \in \mathcal{S}_{ 1}^{N_{ 1}}\cdots \times \mathcal{S}_{k+1}^{N_{k+1}}$ with $ \underline{x}_{k+1}=(x_1, \cdots, x_{N_{k+1}})$ .
 Then we have $$R_{k+1} \leq
 \sum_{s=1}^{N_{k+1}}S_{n_{k+1}}f(x_{s})+ l_{k+1} \eta.$$
 \end{prop}
\begin{proof}
From (\ref{rrk}), one has
 \begin{align}\label{rkgu1}
 \begin{split}
   R_{k+1}\leq& \sup_{x\in H(\underline{x}_1, \cdots, \underline{x}_k, \underline{x}_{k+1} ;  2\epsilon_{k+1} )}S_{l_{k+1}-m_{k+1}}f(x) + 1\\
  \leq & \sup_{x\in B_{l_{k}+R_k}(z(\underline{x}_1, \cdots, \underline{x}_k),   2\epsilon_{k+1} )}S_{l_{k }+R_k}f(x)+ \sum_{s=1}^{N_{k+1}}S_{n_{k+1}}f(x_{s})\\
  &+ N_{k+1}n_{k+1}\text{Var}(f;  2\epsilon_{k+1})+N_{k+1}m_{k+1}\|f\| +1\\
  \leq &  (l_{k}+R_k)\|f\|+ \sum_{s=1}^{N_{k+1}}S_{n_{k+1}}f(x_{s})\\
  &+ N_{k+1}n_{k+1}\text{Var}(f;  2\epsilon_{k+1})+N_{k+1}m_{k+1}\|f\| +1.\\
   \end{split}
 \end{align}
We continue the estimates of the upper bound of $R_{k+1}$.
From (\ref{et4}),
\begin{align}\label{rre1}
(l_{k}+R_k)\|f\|\leq  \frac{l_{k+1} \eta}{4}.
\end{align}
By (\ref{et1}),
 \begin{align}\label{rre2}
 N_{k+1}n_{k+1}\text{Var}(f;  2\epsilon_{k+1})\leq N_{k+1}n_{k+1}\text{Var}(f;  2\epsilon)<\frac{l_{k+1}\eta}{4 }.
\end{align}
By (\ref{et3}),
 \begin{align}\label{rre3}
N_{k+1}m_{k+1}\|f\|= l_{k+1}\frac{N_{k+1}m_{k+1}}{l_{k+1}}\|f\|\leq l_{k+1}\frac{ m_{k+1}}{n_{k+1}}\|f\|\leq \frac{l_{k+1}\eta}{4 }.
\end{align}
By (\ref{et2}),
\begin{align}\label{rre4}
 1   \leq \frac{l_{k+1}}{N_{k+1}}<\frac{l_{k+1} \eta}{4}.
\end{align}
Combing (\ref{rkgu1}),   (\ref{rre1}), (\ref{rre2}), (\ref{rre3}) with (\ref{rre4}), we have
\begin{align*}
\begin{split}
R_{k+1}\leq  & \frac{l_{k+1} \eta}{4}+\sum_{s=1}^{N_{k+1}}S_{n_{k+1}}f(x_{s})+\frac{l_{k+1} \eta}{4}+\frac{l_{k+1} \eta}{4} +\frac{l_{k+1} \eta}{4} \\
=&  \sum_{s=1}^{N_{k+1}}S_{n_{k+1}}f(x_{s})+ l_{k+1} \eta.
\end{split}
\end{align*}
\end{proof}

 \begin{lem}\label{jicheng}
If $z\in \mathcal{C}_{k+1} $ descends from  $x\in \mathcal{C}_k$ then
$$\overline{B}_{l_{k+1}}(z,  \epsilon_{k+1} )\subset \overline{B}_{l_k+R_k}(x, \epsilon_k )\subset \overline{B}_{l_k}(x, \epsilon_k ).$$
\begin{proof}
Since $z\in \mathcal{C}_{k+1} $ descends from  $x\in \mathcal{C}_k$, we have $d_{l_{k}+R_{k}}(z,x)<\epsilon_{k+1}.$
For any $y\in B_{l_{k+1}}(z,  \epsilon_{k+1} ),$
\begin{align*}
\begin{split}
d_{l_{k }+R_{k }}(y, x)&\leq  d_{l_{k }+R_{k }}(y, z)+d_{l_{k }+R_{k }}(z, x)\\
&\leq  \epsilon_{k+1}+ \epsilon_{k+1}  \\
&\leq \epsilon_k.
\end{split}
\end{align*}
So $\overline{B}_{l_{k+1}}(z, \epsilon_{k+1} )\subset \overline{B}_{l_k+R_{k}}(x, \epsilon_k ).$ In addition, $\overline{B}_{l_k+R_k}(x, \epsilon_k )\subset \overline{B}_{l_k}(x, \epsilon_k )$ is obvious.
\end{proof}
\end{lem}
Let $F_k=\bigcup_{x\in \mathcal{C}_k}\overline{B}_{l_k+R_k}(x, \epsilon_k)$. By lemma \ref{jicheng}, $F_{k+1}\subset F_k$. Since we have a decreasing sequence of compact sets, the intersection $$F=\bigcap_kF_k$$ is non-empty.
\begin{lem}
$F\subset  \mathbf D^{x_0 }_f .$
\begin{proof}
 Let $p\in F$. There exists a sequence $\{z_{k}\}_{k\geq1}$ with $z_k\in \mathcal{C}_k$  and $z_k$ descends from $z_{k-1}$ for each $k\geq 2$ satisfying  Lemma \ref{jicheng} such that  $p\in \overline{B}_{l_k+R_k}(z_k, \epsilon_k)$   for each $k\geq 1 $.  For $k\geq 1,$ assume  $z_k=z(\underline{x}_1, \underline{x}_2,\cdots \underline{x}_k)\in \mathcal{C}_k$.
 Hence, we have
 \begin{align*}
 \begin{split}
 d_{R_{k+1}}(T^{l_{k+1}}p, x_0)\leq & d_{R_{k+1}}(T^{l_{k+1}}p, T^{l_{k+1}}z_{k+1})+d_{R_{k+1}}(T^{l_{k+1}}z_{k+1}, x_0)\leq \epsilon_{k+1}+\epsilon_{k+1}=2\epsilon_{k+1}.
 \end{split}
 \end{align*}
 So we have
\begin{align}\label{22222}
 T^{l_{k+1}}p\in B_{R_{k+1}}(  x_0, 2\epsilon_{k+1} ).
 \end{align}
 It follows the construction of the fractal $F$, we can see that $z_k\in H(\underline{x}_1, \cdots, \underline{x}_k ;  \epsilon_{k } )$ for each $k\geq1$.
Assume $\underline{x}_{k+1}=(x_1, \cdots, x_{N_{k+1}})$. Combing with the fact that $p\in \overline{B}_{l_{k }+R_{k}}(z_{k},  \epsilon_{k} )$ for every $k\geq 1$ and
\begin{align*}
 \begin{split}
 &H(\underline{x}_1, \cdots, \underline{x}_{k }, \underline{x}_{k +1} ; \epsilon_{k+1 } ) \\
 =& B_{l_{k}+R_k}(z_{k }, \epsilon_{k+1 } )\bigcap T^{-l_{k }-R_k-m_{k+1}}\Big(
 \cap_{j=1}^{N_{k+1 }} T^{-(j-1)(n_{k+1 }+m_{k +1})}B(x_{ j },\epsilon_{k+1 }  )\Big),
 \end{split}
 \end{align*}
one has
\begin{align}
\begin{split}
d_{l_{k}+R_{k}}(p,z_{k})\leq& d_{l_{k}+R_{k}}(p,z_{k+1})+d_{l_{k}+R_{k}}(z_{k+1},z_{k})\\
\leq& d_{l_{k+1}+R_{k+1}}(p,z_{k+1})+d_{l_{k}+R_{k}}(z_{k+1},z_{k})\\
< & \epsilon_{k+1} + \epsilon_{k+1}=2\epsilon_{k+1},
\end{split}
\end{align}
which implies that
\begin{align}\label{2222}
p\in B_{l_k+R_k}(z_k, 2\epsilon_{k+1}).
\end{align}
Similarly, for each $1\leq j\leq N_{k+1},$
\begin{align}\label{222222}
\begin{split}
&d_{n_{k+1}}(T^{l_{k}+R_{k}+m_{k+1}+(j-1)(n_{k+1}+m_{k+1})}p,x_{j})\\
\leq&d_{n_{k+1}}(T^{l_{k}+R_{k}+m_{k+1}+(j-1)(n_{k+1}+m_{k+1})}p,T^{l_{k}+R_{k}+m_{k+1}+(j-1)(n_{k+1}+m_{k+1})}z_{k+1})\\
& +d_{n_k}(T^{l_{k}+R_{k}+m_{k+1}+(j-1)(n_{k+1}+m_{k+1})}z_{k+1},x_{j})\\
< & \epsilon_{k+1} + \epsilon_{k+1}=2\epsilon_{k+1}.
\end{split}
\end{align}
By (\ref{2222}) and (\ref{222222}), we have
$ p\in H(\underline{x}_1, \cdots, \underline{x}_{k+1} ;  2\epsilon_{k+1 } )$.
By (\ref{rrk}), we have
 \begin{align}\label{222}
 R_k-S_{l_k-m_{k}}f(p)\geq 0.
 \end{align}
It follows from (\ref{22222}) and (\ref{222}),
$$     \dfrac{\max\{t\in\mathbb N:~T^{l_k}(p)\in B_{t}( x_0,2\epsilon_k)\}-S_{l_k}f(p)}{l_k}\geq  \dfrac{ R_k-S_{l_k}f(p)}{l_k}\geq \frac{-m_{k}\|f\|}{l_{k}}   $$ for each $k\geq 1.$
Moreover,
$$     \limsup\limits_{k\to\infty}\dfrac{\max\{t\in\mathbb N:~T^{l_k}(p)\in B_{t}( x_0,2\epsilon_k)\}-S_{l_k}f(p)}{l_k} \geq 0   .$$
Clearly, for any $\epsilon>0$
$$     \limsup\limits_{n\to\infty}\dfrac{\max\{t\in\mathbb N:~T^n(p)\in B_{t}(x_0,2\epsilon)\}-S_nf(p)}{n}\geq 0.$$
Furthermore,
$$    \lim\limits_{\epsilon\to 0}\limsup\limits_{n\to\infty}\dfrac{\max\{t\in\mathbb N:~T^n(p)\in B_{t}(x_0,2\epsilon)\}-S_nf(p)}{n}\geq 0 , $$
which implies that $p\in \mathbf D^{x_0 }_f.$
\end{proof}
\end{lem}

 \subsection{Construction of  a special sequence of measure $\mu_k$}
 \begin{lem}\label{shadowx}
 Let $z_{k}=z(\underline{x}_1, \cdots, \underline{x}_k)\in \mathcal{C}_{k}.$ For all $i\in\{1,\cdots,k\},$ put $z_{i}=z(\underline{x}_1, \cdots, \underline{x}_i)\in \mathcal{C}_{i}.$  Then we have

 $$d_{l_{i} }(z_{k},z_{i})<\epsilon_{i}.$$

\begin{proof}
Set $z_1:=z(\underline{x}_1 ), z_2=z(\underline{x}_1,  \underline{x}_2), z_i=z(\underline{x}_1, \cdots, \underline{x}_i), z_k=z$ for $z=z(\underline{x}_1, \cdots, \underline{x}_k).$
For $1\leq i\leq k-1,$
\begin{align*}
\begin{split}
&d_{l_{i}}(z_{k},z_{i})\\
\leq& d_{l_i}(z_{k},z_{k-1})+d_{l_i}(z_{k-1},z_{k-2})+\cdots+d_{l_i}(z_{i+1},z_{i})\\
\leq& d_{l_{k-1}+R_{k-1}}(z_{k},z_{k-1})+d_{l_{k-2}+R_{k-2}}(z_{k-1},z_{k-2})+\cdots+d_{l_{i}+R_{i}}(z_{i+1},z_{i})\\
< & \epsilon_k + \epsilon_{k-1} +\cdots + \epsilon_{i+1} \leq \epsilon_i.
\end{split}
\end{align*}
\end{proof}
\end{lem}
The following lemma will show the fact that
 $\#\mathcal{C}_k=\#\mathcal{C}_{k-1}\times\#\mathcal{S}_k=\#\mathcal{S}_1\times \cdots \times \#\mathcal{S}_k.$

\begin{lem}\label{jicheng2}
 Let distinct $ (\underline{x}_1, \cdots, \underline{x}_k )$ and  $ (\underline{y}_1, \cdots, \underline{y}_k)\in \mathcal{S}_1^{N_1}\times  \cdots \times\mathcal{S}_k^{N_k}$.
Then $$z(\underline{x}_1, \cdots, \underline{x}_k )\neq z(\underline{y}_1, \cdots, \underline{y}_k).$$
Furthermore, $\mathcal{C}_{j+1}(z)$ is a $(l_{j+1},5\epsilon)$-separated set for each $z\in\mathcal{C}_{j}$ and each $j\geq1.$
 \begin{proof}
 Since  $(\underline{x}_1, \cdots, \underline{x}_{k})\neq (\underline{y}_1, \cdots, \underline{y}_{k}) ,$ There exists $1\leq j\leq k$ such that $\underline{x}_1=\underline{y}_1, \cdots \underline{x}_j=\underline{y}_j, \underline{x}_{j+1}\neq \underline{y}_{j+1}$. Set $p:= z(\underline{x}_1, \cdots, \underline{x}_{j+1} )$ and  $q:=z(\underline{y}_1, \cdots, \underline{y}_{j+1}) $.
Let $\underline{x}_{j+1} =(x_1, \cdots, x_{N_j+1})$ and $\underline{y}_{j+1} =(y_1, \cdots, y_{N_j+1})$.  Assume the $t=\max\{s: x_s\neq y_s\}$ and $c_{j+1}=l_{j+1}-(l_j+R_j+m_{j+1})$
 \begin{align*}
 \begin{split}
 &d_{l_{j+1} }\Big(p, q\Big)\\>& d_{c_{j+1} }\Big(T^{l_j+R_j+m_{j+1}}p, T^{l_j+R_j+m_{j+1}}q\Big)\\
 \geq&   d_{n_{j+1} }\Big(T^{l_j+R_j+m_{j+1}+(t-1)(m_{j+1}+n_{j+1})}p, T^{l_j+R_j+m_{j+1}+(t-1)(m_{j+1}+n_{j+1})}q\Big)\\
 \geq &d_{n_{j+1} }\Big(x_{t} , y_t\Big)-d_{n_{j+1} }\Big(T^{l_j+R_j+m_{j+1}+(t-1)(m_{j+1}+n_{j+1})}p, x_{t}\Big) \\
 &-d\Big(x_t, T^{l_j+R_j+m_{j+1}+(t-1)(m_{j+1}+n_{j+1})}q\Big)\\
  \geq & 6\epsilon - \epsilon_{j+1} - \epsilon_{j+1} \\
  =& 5\epsilon .
\end{split}
\end{align*}
Let $z_j:=z(\underline{x}_1, \cdots, \underline{x}_j )$ for each $j\geq1$, and $\mathcal{C}_{j+1}(z_j)=\{z(\underline{x}_1, \cdots, \underline{x}_j, \underline{x}_{j+1}): \underline{x}_{j+1}\in \mathcal{S}_{j+1}^{N_{j+1}} \}$. It turns out that
$\mathcal{C}_{j+1}(z_j)$ is $(l_{j+1}, 5\epsilon)$ separated set.
Using Lemma \ref{shadowx}, we obtain $z(\underline{x}_1, \cdots, \underline{x}_k )\in B_{l_{j+1}}(p, \epsilon_{j+1})$ and $z(\underline{y}_1, \cdots, \underline{y}_k )\in B_{l_{j+1}}(q, \epsilon_{j+1})$.
Then
\begin{align*}
 \begin{split}
 &d_{l_{j+1} }\Big(z(\underline{x}_1, \cdots, \underline{x}_k ), z(\underline{y}_1, \cdots, \underline{y}_k )\Big)\\
 \geq &d_{l_{j+1} }(p, q)-d_{l_{j+1} }(p,z(\underline{x}_1, \cdots, \underline{x}_k ) )- d_{l_{j+1}}(q,  z(\underline{y}_1, \cdots, \underline{y}_k ))\\
 \geq & 5\epsilon-\epsilon_{j+1}-\epsilon_{j+1}
\\
\geq & 3\epsilon.
\end{split}
\end{align*}
 We finish the proof.
\end{proof}
 \end{lem}

Finally, we define a probability  measure on $F$ which yield the required estimates for the entropy Distribution Principle. For each $z\in \mathcal{C}_k$, we associate a number $\mathcal{L}_k(z)\in (0, \infty)$. Using  the function $f$  as weights, we define, for each $k$, an atomic measure centred on $\mathcal{C}_k$. Precisely, if $z=z(\underline{x}_1, \cdots, \underline{x}_k)$, define
$$\mathcal{L}_k(z): =\mathcal{L} ( \underline{x}_1)\cdots \mathcal{L} ( \underline{x}_k),$$
where if $\underline{x}_i=(x_1^i,\cdots, x_{N_i}^i)\in  \mathcal{S}_i^{N_i},$ then
$$\mathcal{L}(\underline{x}_i):=  \prod_{s=1}^{N_{i}}\exp S_{n_i}(1+f)( x^i_s ).$$
We define
$$\nu_k:=\sum_{z\in\mathcal{C}_k}\delta_{z}\mathcal{L}_k(z) .$$
We normalise $\nu_k$ to obtain a sequence of probability measures $\mu_k$. More precisely, we let $\mu_k :=\frac{1}{\kappa_k}\nu_k$ where $\kappa_k$ is the normalising constant
$$\kappa_k:= \sum_{z\in \mathcal{C}_k}\mathcal{L}_k(z) .$$
\begin{lem}
$\kappa_k=  M_1^{N_1}\cdots  M_k^{N_k}.$
\begin{proof}
Clearly, we have
\begin{eqnarray*}
    \sum_{\underline{x}_i\in \mathcal{S}_i^{N_i}}\mathcal{L}(\underline{x}_i) &=& \sum_{x_1^i \in \mathcal{S}_i} \exp S_{n_i}(f+1)(x_{ 1 }^i)\cdots  \sum_{x_1^i \in \mathcal{S}_i} \exp S_{n_i}(f+1)(x_{ N_{i} }^i) \\
      &=& M_i^{N_i}
  \end{eqnarray*}
 Since each $z\in \mathcal{C}_k$ corresponds uniquely to a sequence $(\underline{x}_1, \underline{x}_2,\cdots, \underline{x}_k).$ We have
  $$\sum_{z\in \mathcal{C}_k}\mathcal{L}_k(z)=\sum_{\underline{x}_1\in \mathcal{S}_1^{N_1}}\cdots \sum_{\underline{x}_k\in \mathcal{S}_k^{N_k}}\mathcal{L}(\underline{x}_1)\cdots \mathcal{L}(\underline{x}_k).$$
  The result follows.
\end{proof}

\end{lem}

\begin{lem}
Suppose that $\nu$ is a limit measure of the sequence of probability measure $\mu_k.$ Then $\nu(F)=1.$
\begin{proof}
Suppose $\mu$ is a limit measure of the sequence of probability measures $\mu_k$. Then $\mu=\lim\limits_{k}\mu_{s_k}$ for some $s_k\to\infty$. For some fixed $s$ and all $p\geq 0$, $\mu_{s+p}( F_s)=1$ since $\mu_{s+p}( F_{s+p})=1$ and $  F_{s+p}\subset   F_{s}$. Therefore, $\mu(  F_{s})\geq \limsup\limits_{k\to\infty}\mu_{s_k}(  F_{s})=1.$ It follows that $\mu( F )=\lim\limits_{s\to\infty}\mu(  F_{s})=1.$
\end{proof}
\end{lem}
 In order to prove Theorem \ref{main}, we give a sequence of lemmas which will allow us to apply the generalised Entropy Distribution Principle. Let $B_{n}(q,  \frac{\epsilon}{2} )$ be an arbitrary ball which intersects $F.$
There exists a sequence $\{z_k\}_{k\geq1}$ with $z_k\in \mathcal{C}_k$  such that $z_i\in B_{n}(q,  \frac{\epsilon}{2} )$ for  $i$ is large enough.
The sequence  $\{z_k\}_{k\geq 1}$ determine the sequence $\{l_k\}_{k\geq 1}\subset \mathbb N$
 and $\{R_k\}_{k\geq 1}\subset \mathbb N $ constructed as above.
Let $k$ be the unique number which satisfies $l_k\leq n< l_{k+1}$.
\begin{lem}\label{le1}
Suppose $\mu_{k+1}(B_n(q,  \frac{\epsilon}{2}))>0$, then     there exists $z_k\in \mathcal{C}_k$  (a unique choice ) such that
$$B_n(q,  \frac{\epsilon}{2} )\cap \mathcal{C}_{k+1}\subset  \mathcal{C}_{k+1}(z_k).$$
\begin{proof}
Suppose that there exist $z'\neq z_k , $ with $z', z_k\in \mathcal{C}_{k}$ satisfy that
$B_n(q,  \frac{\epsilon}{2} )\cap  \mathcal{C}_{k+1}(z')\neq \emptyset $  and $B_n(q,  \frac{\epsilon}{2} )\cap  \mathcal{C}_{k+1}(z_k )\neq \emptyset $. Hence, we can choose $q_1, q_2\in B_{n}(q, \frac{\epsilon}{2})$, and $$d_{l_k}(z_k, z')\leq  d_{l_k}(z_k,  q_1)+d_{l_k}(q_1,q)+d_{l_k}(q, q_2)+d_{l_k}(q_2, z')
 \leq  \epsilon_k +\frac{\epsilon_k}{2}+\frac{\epsilon_k}{2}+ \epsilon_k  =3\epsilon .$$
By the construction, the exists $1\leq i\leq k$, such that $z_i, z'_i\in \mathcal{C}_i$ are 'sisters' and $z_k$ descends from $z_i$, $z'$ descends from $z'_i$.
We observe that
\begin{align*}
\begin{split}
d_{l_i}(z_i, z'_i)\leq &d_{l_i}(z_i,z_{i+1} ) + \cdots d_{l_i}(z_{k-1}, z_k)   \\
&+d_{l_i}(  z_k ,  z ')   + d_{l_i}(z', z'_{k-1})+ \cdots +d_{l_i}(z'_{i+1}, z'_{i})\\
\leq &  \epsilon_{i+1} +\cdots  \epsilon_{k} + 3\epsilon +  \epsilon_{k} +\cdots  \epsilon_{i+1 } \\
\leq& 2\epsilon_{i+1}+ 3\epsilon +2\epsilon_{i+1}\\
< & 5\epsilon,
\end{split}
\end{align*}
which contract the fact that $ \mathcal{C}_i(z_{i-1})$ is $(l_i, 5\epsilon)$ separated set in Lemma \ref{jicheng2}.
Hence, we have $$B_n(q,  \frac{\epsilon}{2} )\cap \mathcal{C}_{k+1}\subset  \mathcal{C}_{k+1}(z_k).$$
\end{proof}
\end{lem}
For the fixed sequence $\{z_k\}_{k\geq 1}$ with $z_{k}\in \mathcal{C}_k$, we set
$z_k=z(\underline{x}_1, \underline{x}_2, \cdots \underline{x}_{k} ) \in \mathcal{C}_k$ with $x_{i}=(x^{i}_{1},x^{i}_{2},\cdots, x^{i}_{N_{i}})\in \mathcal{S}^{N_{i}}_{i}$ for each $1\leq i\leq k.$
Firstly, we estimate $\mathcal{L}_k(z_k)$ as follows.
Since $\mathcal{S}_k\subset Y_k$ for each $k\geq 1,$
\begin{align}\label{che1}
\begin{split}
\mathcal{L}_k(z_k)= &\prod_{i=1}^k\prod_{s=1}^{N_{i}}\exp S_{n_i}(1+f)(x^i_{s} )\\
 \leq & \exp \Big(\sum_{i=1}^k N_in_i(1+\int f d\mu+\eta)  \Big).
 \end{split}
\end{align}
On the other hands, we estimate $\kappa_k$ as follows. By  Proposition \ref{jicheng2},
\begin{align}\label{che2}
\begin{split}
\kappa_k= & M_1^{N_1}\cdots  M_k^{N_k}\\
\geq &\exp\Big( \sum_{i=1}^k N_in_i(h_{\mu}(T)+1+\int f d\mu -\eta)\Big)\\
\geq & \exp \Big( \sum_{i=1}^k N_in_i (C+1)(  1+\int f d\mu) -\sum_{i=1}^k N_in_i\eta \Big).
\end{split}
\end{align}
Let $j\in \{0, \cdots, N_{k+1}-1\}$ be the unique number such that

\noindent {\bf Case 1}  $l_k   \leq n\leq l_k+R_k+m_{k+1} .$

\noindent {\bf Case 2}  $l_k  +R_k+m_{k+1}+j(n_{k+1}+m_{k+1}) <n<l_k+R_k+m_{k+1}+(n_{k+1}+m_{k+1})(j+1).$

\noindent For {\bf Case 1}. By Lemma \ref{le1}, we have
$\nu_{k+1}(B_n(q,   \frac{\epsilon}{2} )) \leq \sum_{z\in \mathcal{C}_{k+1}(z_k)}\delta_{z}\mathcal{L}_{k+1}(z). $
Hence, by (\ref{che1}) and (\ref{che2}),
\begin{align*}
\begin{split}
 \mu_{k+1}(B_n(q,  \frac{\epsilon}{2} ))\leq&\frac{\sum_{z\in \mathcal{C}_{k+1}(z_k)}\delta_{z}\mathcal{L}_{k+1}(z)}{\kappa_{k+1}} \\
 \leq &\frac{\mathcal{L}_{k }(z_k) \prod_{s=1}^{N_{k+1}}\Big(\sum_{x_s^{k+1}\in \mathcal{S}_{k+1} }\exp S_{n_{k+1}}(1+f)(x_s^{k+1}) \Big)}{ M_1^{N_1}\cdots  M_k^{N_k}M_{k+1}^{N_{k+1}}}\\
 \leq &\frac{\mathcal{L}_{k }(z_k)   }{ M_1^{N_1}\cdots  M_k^{N_k} }\\
 \leq & \exp \Big( -\sum_{i=1}^k N_in_i  C(  1+\int f d\mu)  +2\sum_{i=1}^k N_in_i\eta  \Big).
 \end{split}
 \end{align*}
Moreover,  for any $p>1$, we have
$\nu_{k+p}(B_n(q, \frac{\epsilon}{2})) \leq \sum_{z\in \mathcal{C}_{k+p}(z_k)}\delta_{z}\mathcal{L}_{k+p}(z). $
 Then
\begin{align*}
\begin{split}
 &\mu_{k+p}(B_n(q,  \frac{\epsilon}{2} ))\\
 \leq&\frac{\sum_{z\in \mathcal{C}_{k+p}(z_k)}\delta_{z}\mathcal{L}_{k+p}(z)}{\kappa_{k+p}} \\
 \leq &\frac{\mathcal{L}_{k }(z_k) \prod_{j=1}^p\prod_{s=1}^{N_{k+j}}\Big(\sum_{x_s^{k+j}\in \mathcal{S}_{k+j} }\exp S_{n_{k+j}}(1+f)(x_s^{k+j}) \Big)}{ M_1^{N_1}\cdots  M_k^{N_k}M_{k+1}^{N_{k+1}}\cdots M_{k+p}^{N_{k+p}}}\\
 \leq &\frac{\mathcal{L}_{k }(z_k)   }{ M_1^{N_1}\cdots  M_k^{N_k} }\\
 \leq & \exp \Big( -\sum_{i=1}^k N_in_i  C(  1+\int f d\mu)  +2\sum_{i=1}^k N_in_i\eta  \Big).
 \end{split}
 \end{align*}
 Since $x_{j}^i \in Y_i$, for each $1\leq i\leq k $ and $1\leq j\leq N_i$,
 we continue the estimates
 \begin{align*}
 \begin{split}
 & \exp \Big( -\sum_{i=1}^k N_in_i  C(  1+\int f d\mu)  +2\sum_{i=1}^k N_in_i\eta  \Big)\\
 \leq & \exp \Big( -C\sum_{i=1}^k (N_i n_i+\sum_{j=1}^{N_i}S_{n_i}f(x^i_j )-N_in_i\eta  )   +2\sum_{i=1}^k N_in_i\eta  \Big) \\
= & \exp \Big( -C\sum_{i=1}^k\big( N_i  n_i +\sum_{j=1}^{N_i}S_{n}f(x^i_j) \big)   +(2+C)\sum_{i=1}^k N_in_i\eta  \Big).\\
\end{split}
\end{align*}
Moreover, by  Proposition \ref{rkupp}  and (\ref{ll})
\begin{align*}
\begin{split}
&\exp \Big( -C\sum_{i=1}^k\big( N_i  n_i +\sum_{j=1}^{N_i}S_{n}f(x^i_j) \big)   +(2+C)\sum_{i=1}^k N_in_i\eta  \Big)\\
  \leq & \exp \Big( -C\sum_{i=1}^k( N_i  n_i +R_i )  +C\sum_{i=1}^kl_i\eta +(2+C)\sum_{i=1}^k N_in_i\eta  \Big)   \\
  \leq & \exp \Big( -C\sum_{i=1}^k( N_i  n_i +R_i )  +2Cl_k\eta +(2+C)\sum_{i=1}^k N_in_i\eta  \Big)\\
    \leq & \exp \Big( -C\sum_{i=1}^k( N_i  n_i +R_i )  +2Cn\eta +(2+C)n\eta  \Big)\\
    = & \exp \Big( -C\sum_{i=1}^k( N_i  n_i +R_i )    +(2+3C)n\eta  \Big).
 \end{split}
 \end{align*}
Since
 \begin{align*}
\lim\limits_{k\to\infty}\dfrac{\sum_{i=1}^k( N_i  n_i +R_i )}{l_k+R_k+m_{k+1}} \to 1,
\end{align*}
 we have
  \begin{align*}
 \begin{split}
 -C\sum_{i=1}^k( N_i  n_i +R_i )    +(2+3C)n\eta   \leq&   -(C-\eta)(l_k+R_k+m_{k+1})    +(2+3C)n\eta   \\
 \leq&    -(C-\eta)n    +(2+3C)n\eta   \\
 =& -n(C-3\eta-3C\eta).
 \end{split}
 \end{align*}
 Hence,
 \begin{align} \label{cr2}
 \limsup\limits_{k\to\infty}\mu_k(B_n(q, \frac{\epsilon}{2} ))\leq \exp \Big( -n(C-3\eta-3C\eta) \Big)
 \end{align}
 for $n$ large enough.

\noindent For {\bf Case 2}. We need a families of lemmas to estimate the measure.
\begin{lem}
Suppose $\mu_{k+1}(B_n(q,  \frac{\epsilon}{2} ))>0$, then there exists (a unique choice of ) $z_k\in \mathcal{C}_k$ and $ (x_1, \cdots x_j)\in  \mathcal{S}_{k+1}^{j}  $ satisfying
$B_n(q,  \frac{\epsilon}{2} ) \subset \mathcal{A}_{z_{k}; x_1, \cdots x_j}$ where
$$\mathcal{A}_{z_k; x_1, \cdots x_j}:=\{y(z_k;s_1, \cdots  s_j,\cdots s_{N_{k+1}}  )\in\mathcal{C}_{k+1}: s_1=x_1, \cdots  s_j=x_j\} $$
and $y(z_k;s_1, \cdots  s_j,\cdots s_{N_{k+1}}  )$ denotes the point in $\mathcal{C}_{k+1}$ descends from $z_k$ and obtained by $(s_1, \cdots  s_j,\cdots s_{N_{k+1}}).$
Furthermore,
$$\nu_{k+1}(B_n(q,  \frac{\epsilon}{2} ))\leq  \mathcal{L}_k(z_k)\prod_{l=1}^j \exp S_{n_{k+1}}(f+1)(x_{ l}^{k+1})M_{k+1}^{N_{k+1}-j}.$$
\begin{proof}
If $\mu_{k+1}(B_{n}(q, \frac{\epsilon}{2}))>0$, then $\mathcal{C}_{k+1}\cap B_{n}(q, \frac{\epsilon}{2})\neq \emptyset.$ Let $z=y(z_k; x_1, \cdots, x_{N_{k+1}})\in \mathcal{C}_{k+1}\cap B_{n}(q, \frac{\epsilon}{2})$ where $z_k\in \mathcal{C}_{k }$ and $(x_1, \cdots, x_{N_{k+1}})\in \mathcal{S}_{k+1} ^{N_{k+1}}.$ Let
$$\mathcal{A}_{z_k; x_1, \cdots x_j}:=\{y(z_k;s_1, \cdots  s_j,\cdots s_{N_{k+1}}  )\in\mathcal{C}_{k+1}: s_1=x_1, \cdots  s_j=x_j\}.$$
Suppose that $z(x^{'}; a)\in B_{n}(q, \frac{\epsilon}{2})$ with $x^{'}\neq z_{k}, x^{'}\in\mathcal{C}_{k}.$ By the construction, their exists $1\leq i\leq k,$ such that $z_{i},x_{i}'\in\mathcal{C}_{i}$ are sisters and $z_{k}$ descends from $z_{i},x'$ descends from $x_{i}'$. We observe that
 \begin{align*}
\begin{split}
d_{l_{i}}(z_{i},x_{i}')\leq &d_{l_i}(z_{i},z_{i+1})+\cdots+d_{l_i}(z_{k-1},z_{k})\\
&+d_{l_i}(z_{k},q)+d_{l_i}(q,x')+d_{l_i}(x',x_{k-1}')+\cdots+d_{l_i}(x_{i+1}',x_{i}')\\
\leq&\epsilon_{i+1}+\cdots+\epsilon_k+\frac{\epsilon}{2} +\frac{\epsilon}{2}+\epsilon_k+\cdots+\epsilon_{i+1}\\
\leq & 2\epsilon_{i+1}+\epsilon+2\epsilon_{i+1}\\
\leq &3\epsilon,
\end{split}
\end{align*} which contract the fact that $\mathcal{C}_{i}(z_{i-1})$ is $(l_{i},5\epsilon)$ separated set in Lemma \ref{jicheng2}. So
$z_k=x^{'}.$
For $s\in \{1, 2,\cdots, j\}, $ we have
 \begin{align*}
 \begin{split}
 &d_{n_k+1}(T^{l_k+R_k+m_{k+1}+(s-1)(n_{k+1}+m_{k+1})}q, x_{ s}^{k+1}) \\
 \leq &d_{n_k+1}(T^{l_k+R_k+m_{k+1}+(s-1)(n_{k+1}+m_{k+1})}q, T^{(s-1)(n_{k+1}+m_{k+1})}y(z_k; x_1, \cdots, x_{N_{k+1}}))\\
 &+d(T^{(s-1)(n_{k+1}+m_{k+1})}y(z_k; x_1, \cdots, x_{N_{k+1}}), x_{ s}^{k+1})\leq \frac{\epsilon}{2}+\epsilon_{k+1}<2\epsilon.
 \end{split}
 \end{align*}
Since $x_{ s}^{k+1}\in \mathcal{S}_{k+1}$ and  $\mathcal{S}_{k+1}$ is $(n_{k+1}, 6\epsilon)$ separated set, if follows that $s_1=x_1, \cdots, s_j=x_j.$ Then, $B_n(q,  \frac{\epsilon}{2} ) \subset \mathcal{A}_{z_k; x_1, \cdots x_j}$. So
 \begin{align*}
 \begin{split}
 &\nu_{k+1}(B_n(q,  \frac{\epsilon}{2} ))\\
 \leq&\sum_{z\in \mathcal{A}_{z_k; x_1, \cdots x_j}}\mathcal{L}_{k+1}(z)\\
 =& \mathcal{L}_k(z_k)\sum_{\underline{x}_{k+1}: s_1 =x_1, \cdots, s_j =x_j}\mathcal{L}(\underline{z}_{k+1})\\
 \leq &  \mathcal{L}_k(z_k)\prod_{l=1}^j \exp S_{n_{k+1}}(f+1)(x_{ l}^{k+1})M_{k+1}^{N_{k+1}-j}. 
 \end{split}
 \end{align*}
\end{proof}
 \end{lem}
 Without losing generality, we generalise the above lemma as follows (without proof).
 \begin{lem}
 For any $p\geq 1$, suppose $\mu_{k+p}(B_{n}(q, \frac{\epsilon}{2}))>0.$ Let $z_k\in \mathcal{C}_k$ and $x_1, \cdots, x_j$ be as before. Then every $z\in \mathcal{C}_{k+p}\cap B_{n}(q, \frac{\epsilon}{2})$ descends from some point in $\mathcal{A}_{z_k, x_1, \cdots, x_j}.$ We have
 $$\nu_{k+p}(B_n(q, \frac{\epsilon}{2}))\leq  \mathcal{L}_k(z_k)\prod_{l=1}^j \exp S_{n_{k+1}}(f+1)(x_{i_l}^{k+1})M_{k+1}^{N_{k+1}-j}(M_{k+2})^{N_{k+2} }\cdots (M_{k+p})^{N_{k+p} }.$$
 \end{lem}

 \begin{lem}
 For sufficiently large $n$,
 $$\limsup\limits_{k\to\infty}\mu_k(B_n(q, \frac{\epsilon}{2} ))\leq   \exp\Bigg(-n\Big(C-4\eta-3C\eta\Big)\Bigg).$$
   \begin{proof}
By the construction, $M_k\geq  \exp\Bigg(n_k\Big(  \int (f+1) d\mu+h_{\mu}(T)-\eta\Big)\Bigg),$
\begin{eqnarray*}
 &~& \kappa_k(M_{k+1})^{j}\\
   &= & M_1^{N_1}\cdots M_{k}^{N_k}M_{k+1}^{j}\\
   &=& \exp\Bigg( \Big( (C+1)(\int  f   d\mu+1) -\eta\Big) \Big(\sum_{i=1}^kN_in_i+jn_{k+1}\Big)\Bigg).
 \end{eqnarray*}
Accordingly,  similar to (\ref{che1}), we have
\begin{align}
\begin{split}
\mathcal{L}_k(z_k)\prod_{l=1}^j \exp S_{n_{k+1}}(f+1)(x_{ l}^{k+1})\leq & \exp \Bigg(\Big(\sum_{i=1}^k N_in_i+jn_{k+1}\Big)\Big(1+\int f d\mu+\eta\Big)  \Bigg).
\end{split}
\end{align}
Hence,
 \begin{align*}
\begin{split}
 &\mu_{k+p}(B_n(q,  \frac{\epsilon}{2} ))\\
 =&\frac{\nu_{k+p}(B_n(q,  \frac{\epsilon}{2} ))}{\kappa_{k+p}}\\
\leq&\frac{\mathcal{L}_k(z_k)\prod_{l=1}^j \exp S_{n_{k+1}}(f+1)(x_{ l}^{k+1})M_{k+1}^{N_{k+1}-j}(M_{k+2})^{N_{k+2} }\cdots (M_{k+p})^{N_{k+p} }}{M_1^{N_1}\cdots  M_k^{N_k}M_{k+1}^{N_{k+1}}\cdots M_{k+p}^{N_{k+p}}} \\
 =&\frac{\mathcal{L}_k(z_k)\prod_{l=1}^j \exp S_{n_{k+1}}(f+1)(x_{ l}^{k+1})  } {M_1^{N_1}\cdots  M_k^{N_k}M_{k+1}^{j} }\\
 \leq & \frac{\exp \Big((\sum_{i=1}^k N_in_i+jn_{k+1})(1+\int f d\mu+\eta)  \Big)}{ \exp\Big( ( (C+1)(\int  f   d\mu+1) -\eta) (\sum_{i=1}^kN_in_i+jn_{k+1})\Big)} \\
 \leq & \exp\Bigg(-\Big(\sum_{i=1}^k N_in_i+jn_{k+1}\Big)\Big( C (\int  f   d\mu+1)-2\eta\Big) \Bigg)\\
 \leq & \exp\Bigg(-C\Big(\sum_{i=1}^k N_in_i+jn_{k+1}\Big) \Big(  \int  f   d\mu+1  \Big) +2 n\eta\Bigg)\\
 =&\exp\Bigg(-C\Big(\sum_{i=1}^k N_in_i+jn_{k+1}+(\sum_{i=1}^k N_in_i+jn_{k+1})\int f d\mu\Big)+2 n\eta\Bigg).
 \end{split}
 \end{align*}

 Since $x^{i}_{j}\in Y_{i},$ for each $1\leq i\leq k$ and $1\leq j \leq N_{i}$,
 \begin{align*}
\begin{split}
  & \exp\Bigg(-C\Big( \sum_{i=1}^k N_in_i+jn_{k+1} +(\sum_{i=1}^k N_in_i+jn_{k+1})\int f d\mu \Big)   +2n\eta\Bigg)\\
 \leq &\exp\Bigg(-C\Big( \sum_{i=1}^k N_in_i+jn_{k+1} + \sum_{i=1}^{k}\sum_{j=1}^{N_i}S_{n}f(x^{i}_{j})-jn_{k+1}\|f\|\Big)    +2n\eta+Cn\eta\Bigg)\\
  \leq &\exp\Bigg(-C\Big( \sum_{i=1}^k N_in_i+jn_{k+1} + \sum_{i=1}^{k}\sum_{j=1}^{N_i}S_{n}f(x^{i}_{j}) \Big)+C jn_{k+1} \|f\|   +2n\eta+Cn\eta\Bigg).\\
\end{split}
 \end{align*}
 By (\ref{ll1}), we have $\frac{\sum_{i=1}^k N_in_i+jn_{k+1} + \sum_{j=1}^{N_i}R_{i}}{n}\rightarrow1$ and $\frac{C jn_{k+1} \|f\|}{n}\rightarrow0$ as $n\rightarrow\infty$.

 Moreover, by Proposition (\ref{rkupp}) and (\ref{ll}),
 \begin{align*}
 \begin{split}
 &\exp\Bigg(-C\Big( \sum_{i=1}^k N_in_i+jn_{k+1} + \sum_{i=1}^{k}\sum_{j=1}^{N_i}S_{n}f(x^{i}_{j}) \Big)+C jn_{k+1}\|f\|   +2n\eta+Cn\eta\Bigg)\\
 \leq  &\exp\Bigg(-C\Big( \sum_{i=1}^k N_in_i+jn_{k+1} + \sum_{i=1}^{k}R_i - \sum_{i=1}^{k}l_i\eta\Big) +C jn_{k+1}\|f\|   +2n\eta+Cn\eta\Bigg)\\
 \leq &\exp\Bigg(-n(C-\eta)+2l_kC\eta +n\eta+2n\eta+Cn\eta   \Bigg)\\
 \leq  &\exp\Bigg(-n(C-\eta)+2Cn\eta+3n\eta +Cn\eta  \Bigg)\\
 = &\exp\Bigg(-n(C-\eta)+3n\eta+3Cn\eta   \Bigg)\\
 =  &\exp\Bigg(-n \Big( C -4 \eta -3C \eta  \Big)  \Bigg)
 \end{split}
 \end{align*}
for $n$ large enough. Furthermore,
\begin{align}\label{cr1}
\limsup\limits_{k\to\infty}\mu_k(B_n(q, \frac{\epsilon}{2} ))\leq   \exp\Bigg(-n \Big( C -4 \eta -3 C \eta  \Big)  \Bigg)
\end{align}
for $n$ large enough.
\end{proof}
\end{lem}
Applying the Generalised Pressure Distribution Principle to (\ref{cr1}) and (\ref{cr2}), we have
 $$h_{top}^B(T,D_{f}^{x_{0}})\geq h_{top}^B(T,F)\geq P(0, \frac{\epsilon}{2})\geq \min\Big\{C-3\eta-3C\eta, C-4\eta - 3C \eta\Big\}.$$
Letting $\eta\to 0$, we finish the proof.

\section{Applications }
 Recall from \cite{mck}(Proposition 21.2) any topological mixing subshift of finite type satisfies
specification. So by Theorem \ref{main0} and Theorem \ref{main}, we have

\begin{thm}
 Let T be a topological mixing subshift of finite type. Then all results of
Theorem \ref{main0} and Theorem \ref{main} hold.
\end{thm}

Now we present the full symbolic shifts as follows.
\begin{Example}
Given an integer $k>1$; consider the set $\Sigma^{+}_{k}=\{1,\cdots,k\}^{N} $ of sequences
$$\omega=(i_{1}(\omega)i_{2}(\omega)\cdots),$$
where $i_{n}(\omega)\in\{1,\cdots,k\}.$  The shift map $\sigma:\Sigma^{+}_{k}\rightarrow\Sigma^{+}_{k}$ is defined by
$$\sigma(\omega)=(i_{2}(\omega)i_{3}(\omega)\cdots).$$
We denote the symbolic systems by $(\Sigma^{+}_{k},\sigma)$. Now we introduce a distance and thus also
a topology on $\Sigma^{+}_{k}$. Given $\beta>1$, for each $\omega,\omega'\in\Sigma^{+}_{k}$, let
\begin{equation*}
 d(\omega,\omega')=
\begin{cases}\beta^{-|\omega\wedge\omega'|},  &\mathrm{if}~ \omega\neq\omega'; \\[5pt] 0, &\mathrm{if}~\omega=\omega',\\
\end{cases}
\end{equation*}
where $|\omega\wedge\omega'|\in\mathbb{N}$ is the largest positive integer such that $i_{n}(\omega)=i_{n}(\omega'),$ and $|\omega\wedge\omega'|= 0$ if $i_{n}(\omega)\neq i_{n}(\omega')$. Since $\sigma:\Sigma^{+}_{k}\rightarrow\Sigma^{+}_{k}$ is a continuous map of a compact metric space, the topological entropy is well defined. It is well known that $h((\Sigma^{+}_{k},\sigma))=\log k.$ For each $v\in \Sigma^{+}_{k},$  define
\begin{align*}
 \begin{split}
 D^{\alpha}_{v}:&=\{\omega\in\Sigma^{+}_{k}:\limsup_{n\rightarrow\infty}\frac{|\sigma^{n}\omega\wedge v|}{n}\geq\alpha\}\\
  &=\{\omega\in\Sigma^{+}_{k}:\limsup_{n\rightarrow\infty}\frac{-\log_{\beta}d(\sigma^{n}\omega,v)}{n}\geq\alpha\}\\
 &=\{\omega\in\Sigma^{+}_{k}:d(\sigma^{n}\omega,v)\leq M\beta^{-n\alpha} ~for ~any~ constant~ M>1~ and ~infinitely~ n\in\mathbb{N} \}.
 \end{split}
 \end{align*}
By Theorem \ref{main}, we have
$$ h^{B}_{top}(\sigma, D^{\alpha}_{v})=\frac{h_{top}(\sigma)}{1+\alpha}. $$
In the other hands, we have the Hausdorff dimension $dim_{H}(Z)=\frac{h^{B}_{top}(\sigma,Z)}{\log \beta}$ for any
$Z\subset\Sigma^{+}_{k}.$ Moreover, we have
$$dim_{H}(\sigma,D^{\alpha}_{v})=\frac{\log_{\beta}k}
{1+\alpha}.$$
\end{Example}
From the classical uniform hyperbolicity theory, every subsystem restricted on a topological mixing locally maximal hyperbolic set (called basic set or elementary set) satisfies specification property (for example, see \cite{kg}). So by Theorem \ref{main0} and Theorem \ref{main}, we have
\begin{thm}
  Let $f:M \rightarrow M$ be a $C^{1}$ diffeomorphism of a compact Riemannian manifold $M$. Let $T$ be a subsystem restricted on a topological mixing locally maximal hyperbolic set. Then all results of Theorem \ref{main0} and Theorem \ref{main} hold.
  \end{thm}

It is well known that any factor of a topological mixing subshift of finite type has the specification property  and thus our main  theorem applies. Now we give an example of non-uniform hyperbolicity systems as follows.
\begin{Example}{\rm (\cite{Tho1})}
	Fix $I=[0,1]$ and $\alpha\in(0,1)$. The {\bf Manneville-Pomeau }family of maps are given by
	$$f_{\alpha}: I\rightarrow I,  f_{\alpha}(x)=x+x^{1+\alpha}\mod 1.$$
	Considered as a map of $S^{1}$, $f_{\alpha}$ is continuous. Since $f_{\alpha}'(0)=1,$ the system is not uniformly hyperbolic. However, since the Manneville-Pomeau maps are all topologically conjugate to a full shift on two symbols, they satisfy the specification property.
\end{Example}

\end{document}